\documentclass[a4paper,11pt,reqno]{amsart}
\usepackage[centering,includeheadfoot,margin=2.5 cm]{geometry}
\usepackage{graphics,enumitem,epsfig,textcomp,stackrel}
\usepackage{amsfonts,amsmath,amssymb,euscript,color,mathrsfs}
\usepackage[utf8]{inputenc}	

\usepackage[caption=false,subrefformat=parens,labelformat=parens]{subfig}

\usepackage{url}

\usepackage[symbol]{footmisc}

\numberwithin{equation}{section}
\newtheorem{lemma}{Lemma}%[section]
\newtheorem{proposition}{Proposition}%[section]
\newtheorem{remark}{Remark}%[section]

\def\div{\nabla\cdot}

\def\sign{{\rm sign}}

\def\d{\,\mathrm{d}}
\def\eps{\varepsilon}
\def\N{\mathbb{N}}
\def\R{\mathbb{R}}
\def\G{\mathbb{G}}
\def\P{\hbox{\rlap{I}\kern.16em P}}
\def\Q{\hbox{\rlap{\kern.24em\raise.1ex\hbox
      {\vrule height1.3ex width.9pt}}Q}}
\def\M{\hbox{\rlap{I}\kern.16em\rlap{I}M}}
\def\Z{\hbox{\rlap{Z}\kern.20em Z}}
\def\({\begin{eqnarray}}
\def\){\end{eqnarray}}
\def\[{\begin{eqnarray*}}
\def\]{\end{eqnarray*}}
\def\part#1#2{\frac{\partial #1}{\partial #2}}

\def\grad{\nabla}
\def\Norm#1{\left\| #1 \right\|}
\def\pmb#1{\setbox0=\hbox{$#1$}
  \kern-.025em\copy0\kern-\wd0
  \kern-.05em\copy0\kern-\wd0
  \kern-.025em\raise.0433em\box0 }
\def\bar{\overline}

\def\tot#1#2{\frac{\d #1}{\d #2}} 
\def\laplace{\Delta}
\def\d{\,\mathrm{d}}
\def\N{\mathbb{N}}
\def\R{\mathbb{R}}

\def\epsilon{\varepsilon}

\def\I{\mathbb{I}}
\def\A{\mathbb{C}}
\def\P{\mathbb{P}}
\def\En{{E}}

\def \pspace {\bar H^1(\Omega)}

\usepackage{xcolor}

\def\blueJH#1{{\textcolor{blue}{[#1]}}}

\begin{document}
%%%%%%%%%%%%%%%%

%%%%%%%%%%%%%%%%%%%%%%%%%%%%%%%%%
%% Title
%%%%%%%%%%%%%%%%%%%%%%%%%%%%%%%%%
%\title[]{Discrete modelling of biological transport networks}

\centerline{{\Large\textbf{Tensor PDE model of biological network formation}}}
\vskip 7mm

%%%%%%%%%%%%%%%%%%%%%%%%%%%%%%%%%
%% Authors
%%%%%%%%%%%%%%%%%%%%%%%%%%%%%%%%%
\centerline{
	{\large Jan Haskovec}\footnote{Mathematical and Computer Sciences and Engineering Division,
		King Abdullah University of Science and Technology,
		Thuwal 23955-6900, Kingdom of Saudi Arabia; 
		{\it jan.haskovec@kaust.edu.sa}}\qquad
	{\large Peter Markowich}\footnote{Mathematical and Computer Sciences and Engineering Division,
		King Abdullah University of Science and Technology,
		Thuwal 23955-6900, Kingdom of Saudi Arabia;
		{\it peter.markowich@kaust.edu.sa}, and
		Faculty of Mathematics, University of Vienna, Oskar-Morgenstern-Platz 1, 1090 Vienna;
		{\it peter.markowich@univie.ac.at}}\qquad
	{\large Giulia Pilli}\footnote{Faculty of Mathematics, University of Vienna, Oskar-Morgenstern-Platz 1, 1090 Vienna;
		{\it giulia.pilli@univie.ac.at}}
	}
\vskip 10mm

\noindent{\bf Abstract.}
We study an elliptic-parabolic system of partial differential equations describing formation of biological network structures.
The model takes into consideration the evolution of the permeability tensor under the influence of a diffusion term,
representing randomness in the material structure, a decay term describing metabolic cost and a pressure force.
A Darcy's law type equation describes the pressure field.
In the spatially two-dimensional setting,
we present a constructive, formal derivation of the PDE system from the discrete network formation model
in the refinement limit of a sequence of unstructured triangulations.
Moreover, we show that the PDE system is a formal $L^2$-gradient flow
of an energy functional with biological interpretation,
and study its convexity properties.
For the case when the energy functional is convex, we construct
unique global weak solutions of the PDE system.
Finally, we construct steady state solutions in one- and multi-dimensional settings
and discuss their stability properties.
\vskip 7mm

%\allowdisplaybreaks

\noindent{\bf AMSC:} 35G61, 35K57, 35D30, 92C42, 92C42


%\thanks{\textbf{Acknowledgments.} This work is supported by Engineering and Physical Sciences Research Council(EP/K008404/1).}

%\maketitle \centerline{\date}

%\tableofcontents

%%%%%%%%%%%%%%%%%%%%%%%%%%%%%%%%%%%%%%%%
\section{Introduction}
In this paper we study the tensor PDE system modeling formation of biological transportation networks,
\(
   - \div [ (r \I + \A) \grad p ] &=& S, \label{eq:Poisson} \\
   \part{ \A}{t}- D^2 \Delta \A - c^2 \grad p \otimes \grad p + \frac{M'(|\A|)}{|\A|} \A &=& 0,  \label{eq:A}
\)
for the scalar pressure field $p=p(t,x)\in\R$ of the fluid transported within the network and the conductance (permeability) tensor $\A=\A(t,x)\in\R^{d\times d}$
with $d\leq 3$ the space dimension.
 We assume the validity of Darcy's law for slow flow in porous media, connecting the flux $q$ of the fluid with the gradient of its pressure via the action of the
 permeability tensor $\P[\A]$,
\[
   q = - \P[\A] \grad p.
\]
The total permeability tensor is assumed to be of the form $\P[\A] := r\I + \A$, where
the scalar function $r=r(x) \geq r_0 > 0$ describes the isotropic background permeability of the medium, 
while the conductance tensor $\A$ tends to align the flux along its eigenvectors.
%  so that the total effective permeability of the network is assumed to be  $\P[\A] := r\I + \A$.
The source/sink distribution $S=S(x)$ in the mass conservation equation \eqref{eq:Poisson}
is to be supplemented as an input datum and is assumed to be independent of time.
The diffusivity parameter $D\geq 0$ in \eqref{eq:A} measures the effect of random fluctuations in the network,
and the activation parameter $c^2>0$ controls the strength of the network formation feedback loop.
This corresponds to a kinetic term in the respective energy functional
and represents the influence that the fluid movement has on the formation of the network structure.
The last term in \eqref{eq:A} is the functional derivative of the metabolic function $M:\R^+ \to \R^+$, which describes the dependence of the metabolic cost
of maintaining the network on its transportation capacity. 
Here and in the sequel we shall denote
\[
   |\A| := \sqrt{\sum_{i=1}^d \sum_{j=1}^d \A_{ij}^2} = \sqrt{\A:\A}
\]
the Frobenius norm of the matrix $\A$.

%is assumed to be continuously differentiable with uniformly bounded derivative.
We pose \eqref{eq:Poisson}, \eqref{eq:A} on a bounded domain $\Omega\subset\R^d$ with smooth boundary $\partial\Omega$,
subject to homogeneous Dirichlet boundary conditions on $\partial\Omega$ for $\A$
and homogeneous Neumann boundary condition for $p$,
\(   \label{BC_0}
   \A(t,x) = 0,\qquad \nu(x)\cdot\P[\A]\grad p(t,x)=0 \qquad \mbox{for } x\in\partial\Omega,\; t\geq 0,
\)
with $\nu=\nu(x)$ the outer normal vector to $\partial\Omega$,
and subject to the initial condition for $\A$,
\( \label{IC_0}
   \A(t=0,x)=\A^0(x)\qquad\mbox{for } x\in\Omega,
\)
with $\A^0=\A^0(x)$ symmetric and positive semidefinite almost everywhere in $\Omega$.
In order for \eqref{eq:Poisson} to be solvable subject to \eqref{BC_0} we assume $S=S(x)$
to have vanishing mean, i.e.,
\(  \label{ass:S}
   \int_\Omega S(x) \d x = 0.
\)

The system \eqref{eq:Poisson}--\eqref{eq:A} was introduced in \cite{Hu-Cai-19}.
Its fundamental structural property is that it represents the constrained $L^2$-gradient flow associated with the energy functional
\begin{equation} \label{energy}
   \En_D[\A] := \int_{\Omega } \frac{D^2}{2} |\grad \A|^2 + c^2 \grad p[\A] \cdot \P[\A] \grad p[\A]  + M(|\A| ) \d x,
\end{equation}
where $p[\A]$ is the unique solution (up to an additive constant) of the Poisson equation
\eqref{eq:Poisson} subject to the homogeneous Neumann boundary condition \eqref{BC_0}.
In Section \ref{sec:energy} we show that for convex metabolic cost functions $M$
the energy functional is convex.
Based on this fact, we construct unique solutions of  \eqref{eq:Poisson}--\eqref{eq:A},
both with $D>0$ and $D=0$, as the corresponding gradient flows.
Finally, in Section \ref{sec:steady} we study steady states of the system \eqref{eq:Poisson}--\eqref{eq:A}.
For a strictly convex metabolic function $M$ and $D>0$, existence of unique steady states
follows from standard arguments of the calculus of variations.
For $D=0$ and $M$ being a power-law with exponent $\gamma>1$,
the steady state problem reduces to a $p$-Laplacian equation
which, again by standard arguments, has unique solutions.
The situation is significantly different for the borderline case $\gamma=1$,
i.e., for linear metabolic cost. Here one arrives at a highly nonlinear free-boundary problem
separating the region where $\A=0$ from the region with nonvanishing $\A$.
Existence and uniqueness of solutions of this problem was shown
in \cite{HMPS16}.

A variant of the system \eqref{eq:Poisson}--\eqref{eq:A} was proposed in \cite{Hu-Cai},
where the local permeability (conductance) of the porous medium is described in terms
of a vector field $m=m(t,x)\in\R^d$.
The model is obtained by considering \eqref{energy} with $D=0$ and imposing $\A$
to be the rank-one tensor $\A:=m\otimes m$. This leads to the energy functional
\begin{equation} \label{m-energy}
   \widetilde \En_0[m] := \int_{\Omega } c^2 \grad p[m] \cdot (r\I + m\otimes m) \grad p[m]  + M(|m|^2) \d x,
\end{equation}
where $p=p[m]$ is the unique solution of the Poisson equation
\(
   - \div [ (r \I + m\otimes m) \grad p ] = S \label{eq:m-Poisson}
\)
subject to no-flux boundary condition on $\partial\Omega$.
Note that the total permeability tensor $\widetilde\P[m] := r\I + m\otimes m$
has the principal directions (eigenvectors) $m/|m|$ and, resp., $m^\perp$ with eigenvalues
$r(x) + |m|^2$ and, resp., $r(x)$. This means that the flow "feels" only the background permeability $r(x)>0$
in the directions orthogonal to $m$, while the principal permeability is increased by $|m|^2$
in the direction of the conductance vector $m$.
Re-introducing the Dirichlet integral $\frac{D^2}{2} \int_\Omega |\grad m|^2 \d x$ into \eqref{m-energy}
and taking the formal gradient flow with respect to the $L^2$-topology leads to the system
\(
    -\div [(r\I + m\otimes m)\grad p] &=& S,   \label{eq01}\\
    \part{m}{t} - D^2\laplace m - c^2(m\cdot\grad p)\grad p + 2 M'(|m|^2) m &=& 0. \label{eq02}
\)
%The parameter $\mu>0$ is the metabolic coefficient and the metabolic exponent $\gamma\in\R$ is crucial for the type of networks formed;
%see \cite{Hu, AAFM, bookchapter} for the details on the modeling.
Detailed mathematical analysis of the system \eqref{eq01}--\eqref{eq02}
was carried out in the series of papers \cite{HMP15, HMPS16, bookchapter, AAFM}
and in \cite{Li, Xu2, Xu}, while its various other aspects were studied in  \cite{BHMR, HKM, HKM2, HMP19, Hu}.

The system \eqref{eq:Poisson}--\eqref{eq:A} can be seen as a more general model than \eqref{eq01}--\eqref{eq02},
not restricting the network permeability tensor to be a rank-one tensor product.
Moreover, under appropriate assumptions on the parameters,
\eqref{eq:Poisson}--\eqref{eq:A} is obtained as the formal continuum limit of a sequence of discrete graph problems,
as we shall demonstrate in Section \ref{sec:derivation} below.
On the other hand, the system \eqref{eq01}--\eqref{eq02} can be linked to the discrete problem
by considering macroscopic physical laws of fluid flow in porous media,
in particular, Darcy's law and the Carmen-Kozeny equation; see \cite[Section 2.3]{bookchapter}
for details.

The paper is structured as follows.
In Section \ref{sec:derivation} we provide a formal derivation of the tensor PDE model
\eqref{eq:Poisson}--\eqref{eq:A} as a continuum limit of the discrete model on a sequence of unstructured regular graphs.
In Section \ref{sec:energy} we prove the convexity of the energy \eqref{energy}
for convex metabolic cost functions $M$, and construct solutions of the corresponding
gradient flow, both with $D>0$ and $D=0$.
In Section \ref{sec:steady} we study steady states of the system \eqref{eq:Poisson}--\eqref{eq:A},
both with $D>0$ and $D=0$.

%%%%%%%%%%%%%%%%%%%%%%%%%%%%%%%%%%%%%%%%
\newcommand\Vset{\mathbb{V}}
\newcommand\Eset{\mathbb{E}}
\newcommand\E{\mathcal{E}}
\newcommand\CC{\mathcal{C}}
\newcommand\NN{\mathcal{N}}
\newcommand\Qh{\mathbb{Q}^h}
\newcommand\Tup{T^\vartriangle}
\newcommand\Tdown{T^\triangledown}
\newcommand\Th{\mathcal{T}^h}

\section{Formal derivation from the discrete model in 2D}\label{sec:derivation}
A particular instance of the system \eqref{eq:Poisson}--\eqref{eq:A} %with diagonal permeability tensors $C$
with diagonal permeability tensor $\A$ was derived formally in \cite{HKM}
from the discrete network formation model \cite{Hu}.
The system was derived as a formal $L^2$-gradient flow of the continuum
energy functional \eqref{energy}, obtained as a limit
of a sequence of discrete problems under refinement of the network,
with geometry restricted to rectangular networks in 2D.
In \cite{HKM2} the process was carried out rigorously,
in terms of the $\Gamma$-limit of a sequence of discrete energy functionals
with respect to the strong $L^2$-topology.

The goal of this Section is to provide a formal derivation of the continuum energy functional \eqref{energy}
as a limit of discrete energy functionals posed on a sequence of
\emph{unstructured} triangulations of a bounded domain $\Omega\subset \R^2$, 
i.e., without the restriction on rectangular geometries that was imposed in \cite{HKM, HKM2}.
The corresponding rigorous limit procedure will be studied in a future work.

The discrete model is posed on a sequence of undirected graphs $\G^h=(\Vset^h,\Eset^h)$, $h>0$,
consisting of finite sets of vertices $\Vset^h$ and edges $\Eset^h$.
Each edge $(i,j)\in\Eset^h$, connecting the vertices $i$ and $j\in\Vset^h$, has a fixed, prescribed length $L_{ij}=L_{ji}>0$ with $L_{ij} = O(h)$,
where the limit $h\to 0$ is under investigation.
Its conductivity is denoted by $\CC_{ij}=\CC_{ji}\geq 0$.
Each vertex $i\in\Vset^h$ has associated the pressure $P_i\in\R$ of the material flowing through it.
%and we denote $(\Delta P)_{ij}:=P_j-P_i$.
The local mass conservation in each vertex is expressed in terms of the Kirchhoff law
\begin{align}\label{eq:kirchhoff}
   -\sum_{j\in \NN(i)} \CC_{ij}\frac{ P_j-P_i}{L_{ij}}=S_i\qquad \text{for all~}i\in \Vset^h,
\end{align}
where $\NN(i)$ denotes the set of vertices connected to $i\in\Vset^h$ through an edge.
%and $\alpha_{ij}>0$ are weights characterized by geometrical properties of the triangulation.
Moreover, $S=(S_i)_{i\in\Vset^h}$ is the prescribed vector of strengths of flow sources ($S_i>0$) or sinks ($S_i<0$).
Given the vector of conductivities $\CC=(\CC_{ij})_{(i,j)\in\Eset^h}$,  the Kirchhoff law \eqref{eq:kirchhoff}
is a linear system of equations for the vector of pressures $P=(P_i)_{i\in\Vset^h}$.
A necessary condition for its solvability is the global mass balance
\(   \label{mass_balance}
   \sum_{i\in\Vset^h} S_i=0,
\)
whose validity we assume in the sequel.
Then, the linear system \eqref{eq:kirchhoff} is solvable if and only
if the graph with edge weights $\CC=(\CC_{ij})_{(i,j)\in\Eset^h}$ is connected \cite{bookchapter},
where only edges with positive conductivities $\CC_{ij} >0$ are taken into account (i.e., edges with zero conductivities
are discarded). Note that the solution is unique up to an additive constant.

The conductivities $\CC_{ij}$ are subject to a minimization process with the energy functional
\begin{align}\label{eq:energydisc}
   {\En}^h[\CC] := \sum_{(i,j)\in\Eset^h} \left( \CC_{ij} \frac{(P_j-P_i)^2}{L_{ij}^2} + M(\CC_{ij}) \right) L_{ij},
\end{align}
where $\CC = (\CC_{ij})_{(i,j)\in\Eset^h}$ denotes the vector of edge conductivities,
and the sum is taken over unoriented edges, i.e., the edge $(i,j)$ is considered identical with $(j,i)$
and is counted only once.
The first term in the summation describes the (physical) pumping power necessary for transporting the material through the network,
while the second term describes the metabolic cost of maintaining the network structure,
with the metabolic function $M:\R^+\to\R^+$.
It is remarkable that, despite the nonlocal character of the Kirchhoff law \eqref{eq:kirchhoff},
the gradient flow of the energy \eqref{eq:energydisc} constrained by \eqref{eq:kirchhoff}
can be written explicitly in terms of an ODE system for the conductivities $\CC_{ij}$,
\(  \label{ODE}
   \tot{\CC_{ij}}{t} = \left( \frac{(P_j-P_i)^2}{L_{ij}^2} - M'(\CC_{ij}) \right) L_{ij} \qquad \mbox{for }(i,j)\in \Eset^h.
\)
This system, introduced in \cite{Hu-Cai} and studied in more detail in \cite{BHMR},
represents an adaptation model which dynamically responds to local information
and can naturally incorporate fluctuations in the flow.

\begin{figure}
{\centering
\resizebox*{0.5\linewidth}{!}{\includegraphics{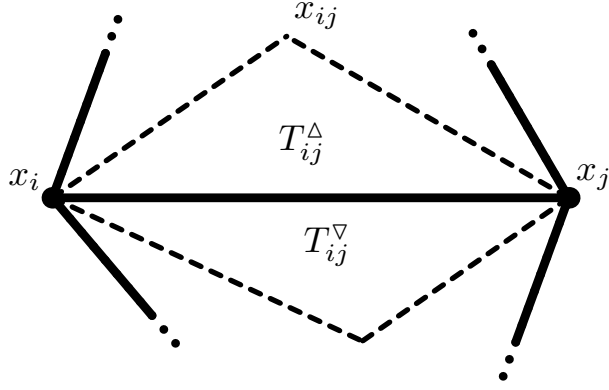}}
\par}
\caption{Edge $e_{ij}\in\Eset^h$ (horizontal solid line) and its adjacent diamond $\diamond_{ij}$
(quadrilateral delimited by dashed line). The diamond consists of the triangle $\Tup_{ij}$,
spanned by the vertices $x_i$, $x_j$ and $x_{ij}$, and the triangle $\Tdown_{ij}$.
\label{fig:1}}
\end{figure}

Let us assume that the graphs $\G^h=(\Vset^h,\Eset^h)$ represent a sequence of finite regular triangulations \cite{Rourke} of a bounded,
polygonal domain $\Omega$ in 2D, and denote $x_i\in\R^2$ the location of the vertex $i\in\Vset^h$.
For convenience, in the sequel we shall use both notations $(i,j)\in\Eset^h$ to address edges by the indices
of their adjacent vertices, and $e_{ij}\in\Eset^h$ to address the linear segment in $\R^2$ connecting
them (i.e., the set of all convex combinations of $x_i$ and $x_j$).
For each (unoriented) edge $e_{ij} \in\Eset^h$ we denote by $\diamond_{ij}\subset\R^2$
the adjacent closed diamond, i.e.,
\[
   \diamond_{ij} := \left\{ x\in\Omega;\, \mbox{dist}(x,e_{ij}) \leq \mbox{dist}(x,\tilde e) \mbox{ for all } \tilde e\in\Eset^h \right\}.
\]
Since the graph $\G^h=(\Vset^h,\Eset^h)$ represents a triangulation of $\Omega$, the diamonds are quadrilaterals;
for edges $e_{ij}$ belonging to the boundary $\partial\Omega$ the definition of $\diamond_{ij}$
has to be adjusted accordingly, such that the corresponding diamond is a quadrilateral too.
For each $i\in\Vset^h$ we denote $U_i$ the union of diamonds adjacent to the vertex $x_i$,
\[
   U_i := \bigcup_{j\in \NN(i)} \diamond_{ij}.
\]
%Note that, with the assumption $L_{ij} \approx h$, %\bigO(h)$, we have
%\(   \label{diamond:reg} \mbox{vol}(\diamond_{ij}) \approx h^2 \qquad\mbox{for all } (i,j)\in\Eset^h. \)
Each diamond $\diamond_{ij}$ consists of two triangles sharing the edge $e_{ij}$,
which we denote $\Tup_{ij}$ and $\Tdown_{ij}$,
see Fig. \ref{fig:1}.
Let us denote $\Th$ the collection of all such triangles, i.e.,
\(   \label{triangulation}
   \Th := \{ \Tup_{ij}; \, (i,j)\in\Eset^h \} \cup \{ \Tdown_{ij}; \, (i,j)\in\Eset^h \}.
\)
Clearly, $\Th$ represents a triangulation of the domain $\Omega$.

The main idea of our construction is to map the vector of discrete edge conductivities $\CC=(\CC_{ij})_{(i,j)\in\Eset^h}$ %on the set of edges $\Eset^h$
onto the set of piecewise constant tensor fields on $\Omega$ through the mapping $\Qh$
defined as
\(    \label{Qh}
   \Qh[\CC](x) := \CC_{ij} \frac{(x_i-x_j) \otimes (x_i-x_j)}{|x_i-x_j|^2} \qquad\mbox{for all } x\in \diamond_{ij}.
\)
I.e., $\Qh[\CC]$ is a constant rank-one tensor on each diamond $\diamond_{ij}$,
with principal direction (eigenvector) parallel to the edge $e_{ij}$ and the corresponding eigenvalue
equal to the conductivity $\CC_{ij}$.
Note that under the assumption $\CC_{ij}>0$ for all $(i,j)\in\Eset$, the tensor $\Qh[\CC]$
is uniformly positive definite on $\Omega$.

Our strategy is, by means of the mapping $\Qh$, to establish a connection between a properly rescaled
version of the discrete energy functional \eqref{eq:energydisc} and its continuum counterpart \eqref{energy}.
Since we only carry out formal arguments in this section, we set $D:=0$ for simplicity,
and also remove the regularization in the permeability tensor $\P[\A]$,
i.e., set $r:=0$ so that $\P[C]=C$. The continuum energy reads then
\( \label{energy:0}
   \E[\A] := \int_\Omega c^2 \grad p[\A] \cdot \A \grad p[\A]  + M(|\A| ) \d x,
\)
where $p[\A] \in H^1(\Omega)$ is the unique solution (up to an additive constant)
of the Poisson equation \eqref{eq:Poisson} with the uniformly positive definite
permeability tensor $\A\in L^\infty(\Omega)$, subject to the homogeneous Neumann
boundary condition \eqref{BC_0}.
Following the strategy of \cite{HKM2}, the connection between the properly rescaled version of the discrete \eqref{eq:energydisc}
and continuum \eqref{energy:0} energy functionals will be established through an "intermediate"
functional $\E^h$, where $p[\A]$ is replaced by a solution of a finite element discretization of the Poisson
equation.

%%%%%%%%%%%%%%%%%%%%%%%%%%%%%%%%%%%%%%%%%%%%%%%
\subsection{Finite element discretization of the Poisson equation and validity of the Kirchhoff law}\label{subsec:FEP}

We consider the first order $H^1$-finite element approximation of the Poisson equation \eqref{eq:Poisson}
with uniformly positive definite permeability tensor $\Qh[\CC]$,
\(   \label{eq:PforFEM}
   - \grad\cdot (\Qh[\CC] \grad p) = S
\)
subject to the no-flux boundary condition on $\partial\Omega$.
We denote $W^h$ the space of continuous, piecewise linear functions on the triangulation $\Th$ introduced in \eqref{triangulation}, i.e.,
\[
   W^h := \left\{ \psi\in C(\Omega);\, \psi \mbox{ linear on each } T\in \Th \right\}.
\]
The finite element discretization of \eqref{eq:PforFEM} reads then
\(   \label{eq:PFEM}
   \int_\Omega \grad \psi^h \cdot \Qh[\CC] \grad p^h \d x = \int_\Omega S \psi^h \d x
      \qquad \mbox{for  all } \psi^h \in W^h.
\)
Using standard arguments (coercivity and continuity of the corresponding bilinear
form) we construct a solution $p^h\in W^h$ of \eqref{eq:PFEM}, unique up to an additive constant;
without loss of generality we fix $\int_\Omega p^h \d x = 0$.
The solution is represented by its vertex values $P_i^h := p^h(x_i)$, $i\in\Vset$.
Moreover, $\grad p^h$ is constant on each triangle $T\in\Th$.

Let us now fix $i\in\Vset$ and denote $\psi_i^h \in W^h$ a test function supported on the union $U_i$ of diamonds adjacent to the vertex $x_i$.
Since $\Qh[\CC]$ is constant on each diamond $\diamond_{ij}$, we have
\[
      \int_\Omega \grad \psi_i^h \cdot \Qh[\CC] \grad p^h \d x
          = \sum_{j\in\NN(i)} \CC_{ij}  \int_{\diamond_{ij}} \grad \psi_i^h \cdot \frac{(x_i-x_j) \otimes (x_i-x_j)}{|x_i-x_j|^2} \grad p^h \d x.
\]
Since $\grad p^h$ is constant on each $T\in\Th$, the Taylor expansion gives
\(   \label{p:Taylor}
   p^h(x_i) - p^h(x_j) = \grad p^h \cdot (x_i - x_j) \qquad \mbox{for all } (i,j) \in\Eset^h,
\)
and, consequently,
\[
   \int_{\diamond_{ij}} \grad\psi_i^h \cdot \frac{(x_i-x_j) \otimes (x_i-x_j)}{|x_i-x_j|^2}\grad p^h \d x
   &=& \int_{\diamond_{ij}}  \grad p^h \cdot \frac{x_i-x_j}{|x_i-x_j|}  \grad\psi_i^h \cdot \frac{x_i-x_j}{|x_i-x_j|} \d x \\
   &=&
   \frac{p^h(x_i)-p^h(x_j)}{|x_i-x_j|}
   \int_{\diamond_{ij}} \grad\psi_i^h \cdot \frac{x_i-x_j}{|x_i-x_j|}  \d x.
\]
Now, for each triangle $T \subset U_i$, we denote its spanning vertices $x_i$, $x_j$ and $x_{ij}$,
and define $\psi_i^h$ by prescribing the vertex values
\[
   \psi^h_i(x_i) = 1, \qquad \psi^h_i(x_j) = 0, \qquad \psi^h_i(x_{ij}) = 0.
\]
This leads to the explicit formula for $\psi^h_i$ on each $T$,
\[
   \psi_i^h(x) := n_{ij}^T\cdot (x-x_j) \qquad\mbox{for } x \in T,
\]
where the vector $n_{ij}^T$ is such that
\[
   n_{ij}^T \cdot (x_{ij} - x_j) = 0, \qquad n_{ij}^T \cdot (x_i-x_j) = 1.
\]
Obviously, such vector exists (and is unique) for any non-collinear triplet $(x_i, x_j, x_{ij})$.
Since $\grad\psi_i^h \equiv  n_{ij}^T$ on each $T$, we have
\[
   \int_T \grad\psi_i^h \cdot \frac{x_i-x_j}{|x_i-x_j|}  \d x 
    = \frac{n_{ij}^T \cdot (x_i-x_j)}{|x_i-x_j|} \mbox{vol}(T)
      = \frac{\mbox{vol}(T) }{L_{ij}}.
\]
Finally, recalling that each $\diamond_{ij}$ consists of two triangles (sharing the edge $e_{ij}\in\Eset^h$), we have
\[
   \int_\Omega \grad\psi_i^h \cdot \Qh[\CC] \grad p^h \d x = \sum_{j\in \NN(i)} \CC_{ij} \frac{p^h(x_i)-p^h(x_j)}{L_{ij}} \frac{\mbox{vol}(\diamond_{ij})}{L_{ij}}.
\]
Therefore, denoting
\(  \label{Si}
   S^h_i := \int_{U_i} S \psi^h_i \d x \qquad\mbox{for all } i\in\Vset,
\)
we have the following proposition.

\begin{proposition}\label{prop:FEM}
Let $S\in L^2(\Omega)$ with $\int_\Omega S\d x = 0$ and let $S_i$ be given by \eqref{Si}.
Let $\CC_{ij}>0$ for all $(i,j)\in\Eset^h$ and denote $P_i^h := p^h(x_i)$,
with $p^h$ being the solution of the discretized Poisson equation \eqref{eq:PFEM}.
Then we have, for all $i\in\Vset$,
\(  \label{Kresc}
     - \sum_{j\in \NN(i)} \CC_{ij} \frac{P_j^h-P_i^h}{L_{ij}} \frac{\mbox{vol}(\diamond_{ij})}{L_{ij}} = S_i.
\)
\end{proposition}

Note that \eqref{Kresc} is a rescaled version of the Kirchhoff law \eqref{eq:kirchhoff}
by the factor $\frac{\mbox{vol}(\diamond_{ij})}{L_{ij}}$.

\subsection{Semi-discrete energy functional}\label{subsec:semi-disc}
For uniformly positive definite permeability tensors $\A\in L^\infty(\Omega)$
we introduce the semi-discrete energy functional
\(    \label{def:EEh}
    \E^h[\A] = \int_\Omega \grad p^h[\A] \cdot \A\grad p^h[\A] + M(|\A|) \d x,
\)
where $p^h[\A]\in W^h$ the unique solution of the FEM-approximation of the Poisson equation \eqref{eq:PFEM}
with the permeability tensor $\A$.
We show that the value of $\E^h$ evaluated at the piecewise constant tensor $\Qh[\CC]$
can be identified with the value of the rescaled version of the discrete energy functional \eqref{eq:energydisc},
given by
\(  \label{def:en:disc:resc}
   \bar E^h[\CC] = \sum_{(i,j)\in\Eset^h} \frac{\mbox{vol}(\diamond_{ij})}{L_{ij}} \left( \CC_{ij} \frac{(P^h_i - P^h_j)^2}{L_{ij}} + M(\CC_{ij}) \right) L_{ij}.
\)

\begin{proposition}\label{prop:En}
%Let $S\in L^2(\Omega)$ satisfying $\int_\Omega S(x) \d x = 0$.
For any vector of positive conductivities $\CC = (\CC_{ij})_{(i,j)\in\Eset^h}$ we have
\(
   \bar E^h[\CC] = \E^h\left[\Qh[\CC] \right],
\)
with the rescaled discrete energy functional $\bar E^h$ defined in \eqref{def:en:disc:resc} and
the semi-discrete energy $\E^h$ given by \eqref{def:EEh}.
\end{proposition}

\begin{proof}
First, we note that, for $x\in \diamond_{ij}$,
\[
   |\Qh[\CC](x)| = \CC_{ij} \left| \frac{(x_i-x_j) \otimes (x_i-x_j)}{|x_i-x_j|^2} \right| = \CC_{ij}.
\]
Consequently, for the metabolic term we have
\[
     \int_\Omega M(|\Qh[\CC]|) \d x
      &=& \sum_{(i,j)\in\Eset^h} \int_{\diamond_{ij}} M(|\Qh[\CC]|) \d x \\
      &=& \sum_{(i,j)\in\Eset^h} \int_{\diamond_{ij}} M(\CC_{ij}) \d x
      = \sum_{(i,j)\in\Eset^h} \mbox{vol}(\diamond_{ij}) \, M(\CC_{ij}),
%      \approx  h^2 \sum_{(i,j)\in\Eset^h} M(\CC_{ij}).
\]
where,
%$|\diamond_{ij}|$ denotes the area of the diamond $\diamond_{ij}$ and, again,
again, the sum is taken over unoriented edges so that each $(i,j) \equiv (j,i)$ is counted only once.

For the pumping term we have
\[
   \int_\Omega \grad p^h \cdot \Qh[\CC]\grad p^h \d x
     = \sum_{(i,j)\in\Eset^h} \CC_{ij} \int_{\diamond_{ij}} \grad p^h \cdot \frac{(x_i-x_j) \otimes (x_i-x_j)}{|x_i-x_j|^2} \grad p^h \d x.
\]
Using \eqref{p:Taylor} again, we have
\[
   \int_{\diamond_{ij}} \grad p^h \cdot \frac{(x_i-x_j) \otimes (x_i-x_j)}{|x_i-x_j|^2} \grad p^h \d x =
   \int_{\diamond_{ij}} \frac{|\grad p^h\cdot (x_i-x_j)|^2}{|x_i-x_j|^2} \d x =
   \mbox{vol}(\diamond_{ij}) \frac{(p^h(x_i) - p^h(x_j))^2}{|x_i-x_j|^2},
\]
so that, denoting again $P_i:=p^h(x_i)$,
\[
   \int_\Omega \grad p^h \cdot \Qh[\CC]\grad p^h \d x
    = \sum_{(i,j)\in\Eset} \CC_{ij} \frac{\mbox{vol}(\diamond_{ij})}{L_{ij}} \frac{(P^h_i - P^h_j)^2}{L_{ij}}.
\]
\end{proof}

Following the strategy of \cite{HKM2}, we would proceed by proving the $\Gamma$-convergence
of the sequence of functionals $\E^h$ towards $\E$ as $h\to 0$.
In \cite{HKM2} this was carried out for the special case of $\G^h$ being a rectangular grid,
%which led to the a-priori restriction on the permeability tensors $C$ to be diagonal.
and the mapping $\Qh$ was constructed to generate piecewise constant diagonal permeability tensors $\A$
with diagonal elements corresponding to the conductivities of the horizontal and vertical edges.
This peculiar construction facilitated the proof of $\Gamma$-convergence
of the sequence $\E^h$ with respect to the norm topology of $L^2(\Omega)$.
However, in our setting, the sequence $\Qh[\CC]$ given by \eqref{Qh} does not converge in the norm topology as $h\to 0$.
Indeed, considering the special case of $\G^h$ being a regular triangulation of $\Omega$,
with, say, $\CC_{ij}=1$ for all $(i,j)\in\Eset^h$, the sequence $\Qh[\CC]$ is a sequence
of periodic functions on $\Omega$, which clearly does not converge in any norm topology.
Instead, we only have weak convergence, which poses significant analytical challenges for proving the $\Gamma$-convergence.
In particular, the so-called weak-strong lemma for the Poisson equation, see \cite[Lemma 1]{HKM2},
which is a fundamental ingredient of the proof, cannot be directly applied. Consequently, new approaches need to be developed
to address this problem, which will be subject of future work.
In this paper we only prove the following formal link between the (rescaled) discrete energy $\bar E^h[\CC]$ and its continuum counterpart $\E$.

\begin{proposition}
Let $\CC = (\CC_{ij})_{(i,j)\in\Eset^h}$ be a vector of positive conductivities on $\Eset^h$.
Then
\(   \label{En:connection}
   \bar E^h[\CC] = \E^h[\Qh[\CC]] = \E[\Qh[\CC]] + O(h^2),
\)
with respect to the $H^1_0(\Omega)$-topology,
with the discrete energy functional $\bar E^h$ defined in \eqref{def:en:disc:resc},
the semi-discrete functional $\E^h$ defined in \eqref{def:EEh} and the continuum energy $\E$ defined in \eqref{energy:0}.
\end{proposition}

\begin{proof}
Since $\CC$ has all positive entries, $\A:=\Qh[\CC]$ is a symmetric, uniformly positive definite tensor field
with $\A\in L^\infty(\Omega)$.
Let us define the bilinear form $B: H^1(\Omega)\times H^1(\Omega) \to \R$,
\[
   B(u,v) = \int_\Omega \grad u \cdot \A \grad v \d x.
\]
Clearly, $B$ induces a seminorm on $H^1(\Omega)$,
\[
   \left|u\right|_{B} := \sqrt{B(u,u)}\qquad\mbox{for } u \in H^1(\Omega).
\]
With this notation we have
\[
   \int_\Omega \grad p \cdot \A \grad p \d x = \left|p\right|_{B}^2.
\]
We now proceed along the lines of standard theory of the finite element method (proof of C\'ea's Lemma in the energy norm, see, e.g., \cite{Ciarlet}).
Due to the Galerkin orthogonality
\(  \label{GO}
   B(p - p^h, \psi) = 0\qquad\mbox{for all }\psi\in W^h,
\)
we have, noting that $p^h\in W^h$,
\[
%   B^N(p^N, p^N) = B^N(p^N-p^h, p^N-p^h) + B^N(p^h, p^h).
   \left|p\right|_{B}^2 - \left|p^h\right|_{B}^2 = \left|p-p^h\right|_{B}^2.
\]
Then, again by \eqref{GO} and by the Cauchy-Schwartz inequality, we have for all $\psi\in W^h$,
\[
   \left|p-p^h\right|_{B}^2 = B(p-p^h, p-\psi)
      \leq \left|p-p^h\right|_{B}  \left| p-\psi \right|_{B},
\] 
and, consequently,
\[
   \left|p-p^h\right|_{B} \leq \inf_{\psi\in W^h} \left|p-\psi\right|_{B}.
\]
Then, by standard results of approximation theory, see, e.g., \cite{Ciarlet}, we have
$\inf_{\psi\in W^h} \left|p-\psi\right|_{B} = O(h)$ and, consequently,
\[
     \E[\A] - \E^h[\A] = \left|p\right|_{B}^2 - \left|p^h\right|_{B}^2 
     = \left|p-p^h\right|_{B}^2 = O(h^2).
\]
An application of Proposition \ref{prop:En} gives the claim \eqref{En:connection}.
\end{proof}

Finally, let us note that both the energy functional \eqref{def:en:disc:resc} and, resp., the Kirchhoff law \eqref{Kresc}
are rescaled by the same factor $\frac{\mbox{vol}(\diamond_{ij})}{L_{ij}}$
with respect to their original definitions \eqref{eq:energydisc} and, resp., \eqref{eq:kirchhoff}.
Consequently, the constrained gradient flow corresponding to the system \eqref{Kresc}, \eqref{def:en:disc:resc}
still admits an explicit description in terms of the ODE
\[
   \tot{\CC_{ij}}{t} = \left( \frac{(P_j-P_i)^2}{L_{ij}^2} - M'(\CC_{ij}) \right) \mbox{vol}(\diamond_{ij}) \qquad \mbox{for }(i,j)\in \Eset^h,
\]
which is a rescaled version of \eqref{ODE}.

%%%%%%%%%%%%%%%%%%%%%%%%%%%%%%%%%%%%%%%%
\section{Energy dissipation structure, convexity and gradient flow}\label{sec:energy}
A fundamental structural property of the system \eqref{eq:Poisson}--\eqref{eq:A}
is that it represents the contrained $L^2(\Omega)$-gradient flow associated with the energy functional
\eqref{energy}.
%$\En: H^1_+(\Omega) \to \R^+$, where $H^1_+(\Omega)$ denotes the set of symmetric, positive semidefinite
%matrix-valued functions with entries in $H^1(\Omega)$, given by
%\begin{equation} \label{energy} \En[\A] := \int_{\Omega } \frac{D^2}{2} |\grad \A|^2 + c^2 \grad p[\A] \cdot \P[\A] \grad p[\A]  + M(|\A| ) \d x,
%\end{equation}
%where $p[\A] \in H^1(\Omega)$ is the unique solution (up to an additive constant)
%of the Poisson equation \eqref{eq:Poisson} with given $\A$, subject to the homogeneous Neumann
%boundary condition \eqref{BC_0}.
The gradient flow structure implies the following energy relation, established in \cite{Hu-Cai-19}.

\begin{proposition}\label{lem:energy}
The energy \eqref{energy} is nonincreasing along smooth solutions of \eqref{eq:Poisson}--\eqref{eq:A} and satisfies
\[
    \tot{}{t} \En_D[\A(t)] = - \int_\Omega \left| \part{\A(t)}{t}(t,x) \right|^2 \d x.
\] 
\end{proposition}

\begin{proof}
The proof is obtained by an obvious modification of \cite[Lemma 1]{HMP15}.
\end{proof}
%For weak solutions built in Theorem~\ref{thm:global_ex} below, we recover the usual weaker form \eqref{energy_ineq}.
%\[    \En[\A(t)] \leq  \En[\A^0].  \]

A major observation is that if the metabolic cost function is convex and nondecreasing,
then the energy functional $\En_D$ is convex.
Let us denote $H^1_+(\Omega)$ the set of symmetric, positive semidefinite
matrix-valued functions with entries in $H^1(\Omega)$.
Obviously, $H^1_+(\Omega)$ is a closed, convex subset of $H^1(\Omega)$.

\begin{proposition}\label{prop:convex}
Let $\Omega$ be a bounded domain in $\R^d$ and
denote $C_\Omega$ the corresponding Poincar\'{e} constant.
Let $M:\R^+ \to \R^+$ %be convex and nondecreasing on $\R^+$.
be such that
\(
   \frac{M'(s)}{s} \geq - D^2 C_\Omega \qquad \mbox{if } M''(s) \geq \frac{M'(s)}{s}, \label{cond:M1} \\
   M''(s) \geq - D^2 C_\Omega \qquad \mbox{if } M''(s) < \frac{M'(s)}{s}. \label{cond:M2}
\)
Then the energy functional $\En_D$ given by \eqref{energy} is convex on $H^1_+(\Omega)$.
Moreover, if \eqref{cond:M1}, \eqref{cond:M2} hold with sharp inequalities on the left-hand side, then $\En_D$ is strictly convex.
\end{proposition}

\begin{proof}
Using $p=p[\A]$ as a test function in the weak formulation of \eqref{eq:Poisson}, one has
\(  \label{p}
   \int_\Omega \grad p \cdot \P[\A] \grad p \d x = \int_\Omega S p \d x,
\)
so that the energy can be written as
\[
   \En_D[\A] = \int_{\Omega } \frac{D^2}{2} |\grad \A|^2 + M(|\A| ) + c^2 S p[\A]  \d x.
\]
By elementary calculation, the second-order variation of the first two terms above %in directions $\Phi, \Psi\in H^1_+(\Omega)$ reads
in direction $\Phi \in H^1_+(\Omega)$ reads
\[
%   D^2 \int_\Omega \grad\Phi : \grad\Psi \d x + \int_\Omega M''(|A|) \left(\frac{A}{|A|}:\Phi \right) \left(\frac{A}{|A|}:\Psi \right) \d x,
   D^2 \int_\Omega |\grad\Phi|^2 \d x + \int_\Omega M''(|A|) \left| \frac{A}{|A|}:\Phi \right|^2 + \frac{M'(|A|)}{|A|} \left( |\Phi|^2 - \left| \frac{A}{|A|}:\Phi \right|^2
   %\frac{|A:\Phi|^2}{|A|^2}
   \right) \d x.
\]
Using $\frac{\delta p[\A,\Phi]}{\delta\A}$ as a test function in the weak formulation of \eqref{eq:Poisson} and calculating the first-order variation
in direction $\Phi \in H^1_+(\Omega)$ gives
\[
    \frac{\delta^2}{\delta\A^2} \int S p[\A;\Phi] \d x  &=& \int_\Omega \grad \frac{\delta p}{\delta\A} \cdot (r \I + \A ) \grad \frac{\delta p}{\delta\A} \d x \\
      && + \int_\Omega  \grad \frac{\delta^2 p}{\delta\A^2} \cdot ( r \I + \A ) \grad p + \grad \frac{\delta p}{\delta\A} \cdot \Phi \grad p \d x,
\]
where here and in the sequel we use the short-hand notation $\frac{\delta p}{\delta\A}$ for $\frac{\delta p[\A,\Phi]}{\delta\A}$.
The first-order variation of the weak formulation of the Poisson equation \eqref{eq:Poisson} in direction $\Phi \in H^1_+(\Omega)$
with test function $\phi\in H^1(\Omega)$ reads
\( \label{convex_relation1}
   \int_\Omega \grad\phi\cdot (r \I + \A) \grad \frac{\delta p}{\delta\A} + \grad\phi\cdot \Phi \grad p \d x = 0,
\) 
and using $\frac{\delta p}{\delta\A}$ as a test function we obtain the identity
\[
   \int_\Omega \grad \frac{\delta p}{\delta\A} \cdot \Phi \grad p  \d x  = - \int_\Omega \grad \frac{\delta p}{\delta\A} \cdot (r \I + \A ) \grad \frac{\delta p}{\delta\A} \d x.
\]
We now again take a variation of \eqref{convex_relation1} in direction $\Phi \in H^1_+(\Omega)$ and use the pressure $p$ as a test function,
which leads to
\[
   \int_\Omega \grad p \cdot (r \I + \A ) \grad \frac{\delta^2 p}{\delta\A^2} \d x = - 2 \int_\Omega \grad p \cdot \Phi \grad \frac{\delta p}{\delta\A} \d x.
\]
Consequently, we arrive at
\[
   \frac{\delta^2}{\delta\A^2} \int_\Omega S p[\A;\Phi] \d x = 2 \int_\Omega \grad \frac{\delta p}{\delta\A} \cdot (r \I + \A) \grad \frac{\delta p}{\delta\A} \d x,
\]
and, finally,
\[
   \frac{\delta^2 \En_D[\A; \Phi]}{\delta\A^2} &=&
      D^2 \int_\Omega |\grad\Phi|^2 \d x + \int_\Omega M''(|A|) \left| \frac{A}{|A|}:\Phi \right|^2 + \frac{M'(|A|)}{|A|} \left( |\Phi|^2 - \left| \frac{A}{|A|}:\Phi \right|^2 \right) \d x \\
      && + 2 c^2 \int_\Omega \grad \frac{\delta p}{\delta\A} \cdot (r \I + \A) \grad \frac{\delta p}{\delta\A} \d x.
\]
Obviously, the last term is nonnegative. Using the Poincar\'{e} inequality for the first term of the right-hand side
and denoting $C_\Omega$ the Poincar\'{e} constant for $\Omega$, we have
\[
   \frac{\delta^2 \En_D[\A; \Phi]}{\delta\A^2} \geq
       \int_\Omega \left( D^2 C_\Omega +  \frac{M'(|A|)}{|A|} \right) |\Phi|^2 \d x
          + \int_\Omega \left( M''(|A|) -  \frac{M'(|A|)}{|A|} \right) \left| \frac{A}{|A|}:\Phi \right|^2 \d x.
\]
Denoting $\Omega^+$ the set where $M''(|A|) -  \frac{M'(|A|)}{|A|} \geq 0$,
and $\Omega^-$ its complement in $\Omega$, we have
\[
   \frac{\delta^2 \En_D[\A; \Phi]}{\delta\A^2} \geq
       \int_{\Omega^+} \left( D^2 C_\Omega +  \frac{M'(|A|)}{|A|} \right) |\Phi|^2 \d x
          + \int_{\Omega^-} \left( D^2 C_\Omega + M''(|A|)  \right) |\Phi|^2 \d x,
\]
where we used $\left| \frac{A}{|A|}:\Phi \right| \leq |\Phi|$.
Clearly, conditions \eqref{cond:M1}, \eqref{cond:M2} imply the convexity of $\En_D$,
which is strict if the inequalities on the left-hand side of \eqref{cond:M1}, \eqref{cond:M2} are sharp.

%Obviously, for convex, nondecreasing $M$ and positive semidefinite $\A$,
%the above expression is nonnegative for all $\Phi \in H^1_+(\Omega)$, which impliex the convexity of $\En_D$.
\end{proof}

%A direct consequence of Proposition \ref{prop:convex} is the fact that if the metabolic cost function
%is $\lambda$-convex, i.e., if $M(s) - \frac{\lambda}{2}s^2$ is convex, then the energy $\En_D$
%is $\lambda$-convex on $H^1_+(\Omega)$, i.e.,
%\(   \label{def:lambda-convex}
%   \A \mapsto \En_D[\A] - \frac{\lambda}{2} \int_\Omega |\A|^2 \d x \qquad\mbox{is convex.}
%\)

\begin{remark}
Let us consider the generic case of $M$ being the power-law $M(s) = s^\gamma/\gamma$
with exponent $\gamma>0$, as discussed in \cite{Hu-Cai} and the subsequent works.
Then $M'(s)/s = s^{\gamma-2}$, so that condition \eqref{cond:M1} is implicitly verified for all $s>0$.
Condition \eqref{cond:M2} is verified if and only if $\gamma\geq 1$.
Consequently, $\En_D$ is strictly convex for every $\gamma\geq 1$ and $D>0$.
For $D=0$ it is strictly convex for $\gamma>1$ and convex for $\gamma=1$.
\end{remark}

%%%%%%%%%%%%%%%%%%%%%%%%%%%%%
\subsection{Existence and uniqueness of gradient flow}\label{subsec:main}
For a metabolic function $M$ verifying the condition \eqref{cond:M1}, \eqref{cond:M2}, we prove the existence and uniqueness
of weak solutions of the system \eqref{eq:Poisson}--\eqref{eq:A}
by an application of the standard theory of convex gradient flows on Hilbert spaces, see, e.g., \cite{AGS, Santambrogio}.
Before we proceed with the proof, we establish two auxiliary lemmas.
In the sequel we shall denote by $\pspace$ the closed subspace of $H^1(\Omega)$
of functions with vanishing mean.

\begin{lemma}\label{lem:Poisson}
Let ${\A} \in L^2(\Omega)^{d\times d}$ be a matrix-valued function,
positive semidefinite almost everywhere on $\Omega$,
and let $S\in L^2(\Omega)$ with $\int_\Omega S(x) \d x = 0$.
%such that $\tilde{\A}(x) \geq 0$ e.a. in $\Omega$.
Then the Poisson equation\eqref{eq:Poisson}
admits a unique weak solution $p \in \pspace$.
Moreover, it holds
\( \label{est:Poisson}
	\Norm{\grad p}_{L^2(\Omega)} \leq C_\Omega \Norm{S }_{L^2(\Omega)},
\)
where $C_\Omega$ denotes the Poincar\'{e} constant on $\Omega$.
\end{lemma}

\begin{proof}
See \cite[Lemma 6]{HMP15}.
\end{proof}

\begin{lemma} \label{lem:sqrt}
%Let $\Omega\subset\R^n$ be a bounded domain in $\R^n$.
Let $(\A_k)_{k\in\N}$ be a sequence of %almost everywhere
symmetric and positive semidefinite matrix-valued
functions on a bounded domain $\Omega\subset\R^d$.
%such that for almost all $x\in\Omega$ and each $k\in\N$ the matrix $\A_k(x)$ is symmetric and positive semidefinite.
Let the sequence $\A_k$ converge to $\A$ in the norm topology
of the space $L^2(\Omega)^{d\times d}$ as $k\to\infty$.
Then there exists a subsequence of %$\left(\sqrt{\A_k}\right)_{k\in\N}$
$\sqrt{\A_k}$ converging to $\sqrt{\A}$ in the norm topology of $L^4(\Omega)^{d\times d}$.
\end{lemma}

\begin{proof}
As the sequence $\A_k$ converges in the norm topology of $L^2(\Omega)$ to $\A$,
there exists a subsequence converging almost everywhere on $\Omega$ to $\A$.
As the operation of taking the principal square root is continuous on the set of positive
semidefinite matrices, see p. 411 of \cite{RogerJohnson}, a subsequence of $\sqrt{\A_k}$
converges pointwise almost everywhere to $\sqrt{\A}$.
	
Symmetry of $\A_k$ implies symmetry of $\sqrt{\A_k}$, giving
\(  \label{eq:sqrtA}
    \left| \sqrt{\A_k} \right|_2 = \varrho\left(\sqrt{\A_k}\right) = \sqrt{\varrho(\A_k)} = \sqrt{|\A_k|_2},
\)
where $|\cdot|_2$ denotes the spectral norm and $\varrho(\cdot)$ the spectral radius.
Consequently, the boundedness of $\A_k$ in $L^2(\Omega)^{d\times d}$ implies boundedness
of $\sqrt{\A_k}$ in $L^4(\Omega)^{d\times d}$
%\begin{equation}  \label{eq1}
%	\|{\sAk}\|_{L^4(\Omega)}^4 = \int_\Omega |\sAk|_2^4 \d x = \int_\Omega |A_k|_2^2 \d x = \|{A_k}\|_{L^2(\Omega)}^2.
%\end)
and there is a subsequence of $\sqrt{\A_k}$ converging weakly in $L^4(\Omega)^{d\times d}$ to $\sqrt{\A}$;
identification of the limit is due to the almost everywhere convergence.
Moreover, \eqref{eq:sqrtA} implies that
\[
    \lim_{k\to\infty}  \Norm{ \sqrt{\A_k} }_{L^4(\Omega)}
       = \lim_{k\to\infty} \sqrt{ \Norm{ \A_k }_{L^2(\Omega)}}
       =  \sqrt{ \Norm{\A}_{L^2(\Omega)}} 
       = \Norm{ \sqrt{\A} }_{L^4(\Omega)}.
\]
By uniform convexity of $L^4(\Omega)^{d\times d}$, Proposition 3.32 of \cite{Brezis} implies then the
strong convergence of $\sqrt{\A_k}$ to $\sqrt{\A}$ in $L^4(\Omega)^{d\times d}$.
\end{proof}

In the sequel, let us denote $H_{0,+}^1(\Omega)$ the space of symmetric, positive semidefinite
tensor-valued functions with entries in $H^1(\Omega)$ and vanishing traces on $\partial\Omega$.

\begin{proposition}[Existence and uniqueness of gradient flows] \label{prop:GF}
Let $r>0$, $D>0$, $S\in L^2(\Omega)$ satisfying \eqref{ass:S}
and $\A^0 \in H_{0,+}^1(\Omega)$, symmetric and positive semidefinite almost everywhere in $\Omega$,
and such that $\En_D[\A^0] < \infty$.
Moreover, let the metabolic cost function $M:\R^+\to\R^+$ %be nondecreasing and $\lambda$-convex, with $M(0)=0$, and
verify the conditions \eqref{cond:M1} and \eqref{cond:M2}.
%\(  \label{ass:M2} M'(s)/s \mbox{ locally bounded for }  s\in [0,\infty). \)

Then the problem \eqref{eq:Poisson}--\eqref{IC_0} admits a unique global weak solution $\A \in H^1((0,\infty); L^2(\Omega)) \cap L^2((0,\infty); H^1_{0,+}(\Omega))$
with $\En_D[\A]\in L^\infty(0,\infty)$.
Moreover, $\A=\A(t,x)$ is symmetric positive semidefinite for almost all $x\in\Omega$ and $t\geq 0$,
and satisfies the energy inequality
\(    \label{energy_ineq}
    \En_D[\A(t)] + \int_0^t \int_\Omega \left| \part{\A}{t}(s,x)\right|^2  \d x\d s \leq \En_D[\A^0] \qquad\mbox{for all } t \geq 0.
\)
\end{proposition}

\begin{proof}
The energy functional \eqref{energy} is proper and,
%as shown in Section \ref{sec:energy}, $\lambda$-convex in the sense of \eqref{def:lambda-convex}.
according to Proposition \ref{prop:convex}, convex on $H^1_{0,+}(\Omega)$.
Lower semicontinuity with respect to the $H^1(\Omega)$-topology is obvious for the diffusive term
$\int_{\Omega } \frac{D^2}{2} |\grad \A|^2 \d x$ and the metabolic term $\int_\Omega M(|\A| ) \d x$.
To show lower semicontinuity of the activation term $ \int_\Omega \grad p[\A] \cdot \P[\A] \grad p[\A]  \d x$,
we choose a sequence $\A_k\in L^2(\Omega)$ converging to $\A \in L^2(\Omega)$.
Denoting $p_k$ the unique weak solution of the Poisson equation \eqref{eq:Poisson} with permeability tensor $\A_k$,
constructed in Lemma \ref{lem:Poisson}, and using in as a test function, we have
\[
     \int_\Omega r | \grad p_k |^2 + \grad p_k \cdot \A_k \grad p_k \d x =
         \int_\Omega S p_k \d x.
\]
Then, we estimate, using the Cauchy-Schwartz and Poincar\'{e} inequalities,
\[
   \int_\Omega \left| \sqrt{\A_k} \grad p_k \right|^2 \d x = \int_\Omega \grad p_k \cdot {\A_k} \grad p_k \d x
   &\leq& \int_\Omega S p_k \d x \\
   &\leq& \Norm{S}_{L^2(\Omega)} \Norm{p_k}_{L^2(\Omega)} \\
   &\leq& C_\Omega \Norm{S}_{L^2(\Omega)} \Norm{\grad p_k}_{L^2(\Omega)},
\]
where $C_\Omega$ is the Poincar\'{e} constant for $\Omega$.
Due to the bound \eqref{est:Poisson}, we obtain uniform boundedness of the sequence
$\sqrt{\A_k} \grad p_k$ in $L^2(\Omega)$.
%of the terms $\int_\Omega \grad p_k \cdot {\A_k} \grad p_k \d x = \int_\Omega \left| \sqrt{\A_k} \grad p_k \right|^2 \d x$.
Consequently, up to a possible extraction of a subsequence,
$\sqrt{\A_k} \grad p_k$ convergres weakly in $L^2(\Omega)$.
The limit is identified due to the strong convergence of $\sqrt{\A_k}$
to $\sqrt{\A}$ provided by Lemma \ref{lem:sqrt} and the weak convergence of $\grad p_k$
to $\grad p$.
Then, the lower-semicontinuity of the $L^2$-norm gives
\[
   \int_\Omega \left| \sqrt{\A\ast\eta} \grad p \right|^2 \d x \leq
     \liminf_{k \to \infty} \int_\Omega \left| \sqrt{\A_k\ast\eta} \grad p_k \right|^2 \d x.
\]
Consequently, we have the lower semicontinuity for the energy functional,
\[
   \En[\A] \leq \liminf_{k\to\infty} \En[\A_k].
\]
Then, by the Rockafellar theorem, the subdifferential $\partial\En$ is a maximal monotone operator on $H^1_{0,+}(\Omega)$,
which in turn implies the existence and uniqueness of solutions of the system \eqref{eq:Poisson}--\eqref{IC_0},
%follows then from standard theory of convex gradient flows,
see, e.g., \cite{Brezis}.

To check for preservation of positive definiteness, let us fix an arbitrary vector $\xi\in\R^d$
and define $u(t,x) := \xi\cdot\A(t,x)\xi$. From \eqref{eq:A} we obtain
\[
   \part{u}{t} - D^2 \Delta u + \frac{M'(|\A|)}{|\A|} u = |\xi\cdot \grad p|^2 \geq 0,
\]
and the maximum principle for linear parabolic problems (see, e.g., \cite{Evans})
implies that $u\geq 0$ if $u(t=0)\geq 0$.
% and the term ${M'(|\A|)}/{|\A|}$ is locally bounded.
% \blueJH{if not locally bounded, cut-off and pass to the limit}
Consequently, the positive semidefiniteness of the initial datum $\A^0$ is preserved by
\eqref{eq:A}.
\end{proof}

Let us note that the proof of Proposition \ref{prop:GF} can be easily adapted to the case $D=0$,
i.e., no diffusion, by changing the topology to $L^2(\Omega)$. Then, we obtain the following result.

\begin{proposition}
Let $r>0$, $S\in L^2(\Omega)$ satisfying \eqref{ass:S}
and $\A^0 \in L^2(\Omega)$, symmetric and positive semidefinite almost everywhere in $\Omega$,
and such that $\En_0[\A^0] < \infty$.
Let the metabolic cost function $M:\R^+\to\R^+$
%be nondecreasing and convex and verify condition \eqref{ass:M2}.
verify the conditions \eqref{cond:M1} and \eqref{cond:M2}.

Then the problem \eqref{eq:Poisson}--\eqref{IC_0} with $D=0$
admits a unique global weak solution $\A \in H^1((0,\infty); L^2(\Omega))$,
symmetric and positive semidefinite for almost all $x\in\Omega$ and $t\geq 0$,
with $\En_0[\A]\in L^\infty(0,\infty)$ and satisfying the energy inequality \eqref{energy_ineq}
with $D=0$.
\end{proposition}

\begin{proof}
Convexity of $\En_0$ follows from Proposition \ref{prop:convex},
with conditions \eqref{cond:M1}, \eqref{cond:M2} verified due to
the monotonicity and convexity of $M$.
Lower semicontinuity of $\En_0$ with respect to the norm $L^2$-topology
follows along the steps of the proof of Proposition \ref{prop:GF}.
\end{proof}

%%%%%%%%%%%%%%%%%%%%%%%%%%%%%
\section{Steady states}\label{sec:steady}

In this section we discuss existence and uniqueness of steady states
of the system \eqref{eq:Poisson}--\eqref{eq:A}.% with convex metabolic function $M$.
We first provide a result for the case $D>0$, and then carry out
a more explicit construction for the case $D=0$.
Let us recall that we denoted $H_{0,+}^1(\Omega)$ the space of symmetric, positive semidefinite
tensor-valued functions with entries in $H^1(\Omega)$ and vanishing traces on $\partial\Omega$.

\begin{proposition}
Let $r>0$, $D>0$ and $S\in L^2(\Omega)$ satisfying \eqref{ass:S}.
Moreover, let the metabolic cost function $M:\R^+\to\R^+$ be nonnegative and continuous.
Then there exists a steady state of the system \eqref{eq:Poisson}--\eqref{eq:A} in $H^1_{0, +}(\Omega)$.

Moreover, if $M$ satisfies conditions \eqref{cond:M1}, \eqref{cond:M2},
then the gradient flows constructed in Proposition \ref{prop:GF}
%which are the unique global weak solutions of the system \eqref{eq:Poisson}--\eqref{IC_0},
converge towards a steady state as $t\to\infty$.

The steady state is unique if $M$ satisfies conditions \eqref{cond:M1}, \eqref{cond:M2} with sharp inequalities
on the left-hand side.
\end{proposition}

\begin{proof}
We apply the direct method of calculus of variations on the closed, convex set $H^1_{0, +}(\Omega)$.
Since the energy $\En_D$ given by \eqref{energy} is, by definition, nonnegative,
we have an infimising sequence $(\A_n)_{n\in\N} \subset H^1_{0, +}(\Omega)$.
With $D>0$ and the Poincar\'{e} inequality, the energy $\En_D$ is coercive on $H^1_{0, +}(\Omega)$
due to the presence of the Dirichlet integral.
Consequently, there exists a subsequence of $(\A_n)_{n\in\N}$ converging weakly
to some $\A\in H^1_{0, +}(\Omega)$ and, due to compact embedding, strongly in $L^2(\Omega)$.
Clearly, the Dirichlet integral $\frac{D^2}{2} \int_{\Omega } |\grad \A_n|^2 \d x$ in the definition of $\En_D$
is weakly lower semicontinuous.
Due to the strong convergence of $\A_n$ in $L^2(\Omega)$ and the assumed continuity of $M$,
we can pass to the limit in the metabolic term $\int_\Omega M(|\A_n|) \d x$.
Finally, we refer to the proof of Proposition \ref{prop:GF} where
lower semicontinuity of the activation term $ \int_\Omega \grad p[\A_n] \cdot \P[\A_n] \grad p[\A_n]  \d x$
was shown with respect to the $L^2(\Omega)$ topology.
Consequently, we have
\[
   \En_D[\A] \leq \liminf_{n\to\infty} \En_D[\A_n], 
\]
and $\A\in H^1_{0, +}(\Omega)$ is a global minimizer of $\En_D$,
i.e., a steady state of the system \eqref{eq:Poisson}--\eqref{eq:A}.
The minimizer is unique if $\En_D$ is strictly convex,
which, according to Proposition \ref{prop:convex},
is the case when conditions \eqref{cond:M1}, \eqref{cond:M2} hold with sharp inequalities
on the left-hand side.

Finally, the convergence of the gradient flow towards a steady state
follows from the convexity and coercivity of $\En_D$ by standard theory, see, e.g., \cite{AGS}.
\end{proof}

%\subsection{Steady states for $D=0$}
Next, let us study the case $D=0$.
Then, the stationary version of \eqref{eq:A} reads
\(   \label{steady:A}
   c^2 \grad p \otimes \grad p = m(|\A|) \A,
\)
where we introduced the notation $m(s) := M'(s)/s$. Consequently, $\A$ is a multiple of $\grad p\otimes\grad p$,
\[
   \A = \lambda \grad p\otimes\grad p,
\]
and inserting into \eqref{steady:A}, noting that $|\grad p\otimes\grad p| = |\grad p|^2$, gives
\[
   m(|\lambda| |\grad p|^2) = \frac{c^2}{\lambda}.
\]
Assuming invertibility of $m$, we have
\[
   |\grad p|^2 = \frac{1}{|\lambda|} m^{-1} \left(\frac{c^2}{\lambda} \right).
\]
Consequently, denoting $g(\lambda) := \frac{1}{|\lambda|} m^{-1} \left(\frac{c^2}{\lambda} \right)$ and assuming its invertibility,
we arrive at $\lambda = g^{-1}(|\grad p|^2)$ and
\[
   \A =  g^{-1}(|\grad p|^2) \grad p \otimes \grad p,
\]
where $p$ solves the nonlinear Poisson equation
\[
   - \grad\cdot \left( \left[r + g^{-1}(|\grad p|^2) |\grad p|^2 \right] \grad p \right) = S.
\]

Let us now consider the generic case of $M$ being the power-law $M(s) = s^\gamma/\gamma$
with exponent $\gamma>0$, as discussed in \cite{Hu-Cai} and the subsequent works.
Then, the functions $m=m(s)$ and $g=g(\lambda)$, introduced above,
are both invertible if and only if $\gamma>1$, and we easily calculate
\[
   g^{-1}(\lambda) = c^\frac{2}{\gamma-1} \lambda^{-\frac{\gamma-2}{\gamma-1}}.
\]
Inserting this expression into the Poisson equation \eqref{eq:Poisson}, we obtain the $p$-Laplacian equation
\( \label{steady:pL}
   - \grad\cdot \left( \left[r + c^\frac{2}{\gamma-1} |\grad p|^\frac{2}{\gamma-1} \right] \grad p \right) = S,
\)
subject to the homogeneous Neumann boundary condition on $\partial\Omega$.

%We define the space $\bar H^1(\Omega) := \left\{u\in H^1(\Omega);\, \int_\Omega u \d x = 0 \right\}$.

\begin{proposition}
For any $S\in L^2(\Omega)$ and $\gamma>1$ there exists a unique weak solution
$p\in \bar H^1(\Omega) \cap W^{1,\frac{2\gamma}{\gamma-1}}(\Omega)$
of \eqref{steady:pL} subject to the homogeneous Neumann boundary condition on $\partial\Omega$.
\end{proposition}

\begin{proof}
We define the functional $\mathcal{F}: \bar H^1(\Omega) \to \R \cup \{+\infty\}$,
\(   \label{eq:F}
   \mathcal{F}[p] := \int_\Omega r |\grad p|^2 + c^\frac{2}{\gamma-1} |\grad p|^\frac{2\gamma}{\gamma-1} \d x
      - \int_\Omega p S\d x,
\)
and $\mathcal{F}[p]:=+\infty$ if $\grad p\notin L^\frac{2\gamma}{\gamma-1}$.
Then $\mathcal{F}$ is uniformly convex since $\frac{2\gamma}{\gamma-1} > 2$ for $\gamma>1$.
Coercivity follows from the Poincar\'{e} inequality on $\bar H^1(\Omega)$.
Then the classical theory (see, e.g., \cite{Evans}) provides the existence of a unique minimizer for \eqref{eq:F}.
We conclude by observing that \eqref{steady:pL} is the Euler-Lagrange equation corresponding to the minimization problem \eqref{eq:F}.
\end{proof}

%\subsection{$\gamma=1$}
The situation is significantly different for the case $M(s) = s$, i.e., $\gamma=1$.
Then $m(s)=s^{-1}$ and $g(\lambda)\equiv c^{-2}$, thus not invertible.
In fact, the stationary version of \eqref{eq:A} reads
\[
   c^2 \grad p \otimes \grad p = \frac{\A}{|\A|},
\]
which implies that for almost all $x\in\Omega$ either $\A(x)=0$ or $c^2 |\grad p(x)|^2 =1$.
In the latter case, $\A = \lambda \grad p \otimes \grad p$ with $\lambda>0$.
Consequently, similarly as in \cite[Remark 7]{HMP15}, there exists
a positive, measurable function $\lambda=\lambda(x)$ such that
\[
   \A(x) = \lambda(x) \chi_{\{c|\grad p|=1\}}(x) \grad p(x) \otimes \grad p(x),
\]
where $\chi_{\{c|\grad p|=1\}}$ denotes the characteristic function of the set $\{x\in\Omega;\, c|\grad p(x)|=1 \}$,
and $p=p(x)$ solves the highly nonlinear Poisson equation
\(   \label{nonlinPL}
   - \grad\cdot \left( \left[r + \frac{\lambda}{c^2} \chi_{\{c|\grad p|=1\}} \right] \grad p \right) = S.
\)
subject to the homogeneous Neumann boundary condition on $\partial\Omega$.
This problem was studied in \cite[Section 4.2]{HMPS16}, where it was shown that
its solutions can be constructed as the free-boundary problem
\(
   -\grad\cdot\left[(1+a(x)^2)\grad p\right] &=& S,\qquad p \in H^1_0(\Omega), \label{NLP1} \\
   c^2 |\grad p(x)|^2 &\leq& 1,\qquad \mbox{a.e. on }\Omega, \label{NLP2} \\
   a(x)^2 \left[c^2|\grad p(x)|^2-1 \right] &=& 0,\qquad \mbox{a.e. on }\Omega, \label{NLP3}
\)
for some measurable function $a^2=a(x)^2$ on $\Omega$ which is the Lagrange multiplier
for the condition \eqref{NLP2}.
The function $\lambda=\lambda(x)$ can be chosen as $\lambda(x):=ca(x)$.

Moreover, in \cite[Section 4.2.1]{HMPS16}, it was shown that
solutions of \eqref{NLP1}--\eqref{NLP3} are minimizers of the energy functional
\(   \label{functional_g1}
    \mathcal{J}[p]:=\int_\Omega \left( \frac{|\grad p|^2}{2} - Sp  \right) \d x
\)
on the set $\mathcal{M}:=\{p\in H^1(\Omega), c^2|\grad p|^2\leq 1 \mbox{ a.e. on } \Omega\}$.
Existence and uniqueness of minimizers was established in \cite[Lemma 5]{HMPS16}.
An alternative approach is to consider the penalized problem
\[   \label{penalizedPoisson}
   -\grad\cdot \left[\left( 1+ \frac{(|\grad p_\eps|^2-1/c^2)_+}{\eps}\right) \grad p_\eps \right]= S, \qquad p_\eps \in H^1(\Omega),
\]
where $A_+ := \max(A,0)$ denotes the positive part of $A$, and pass to the limit $\eps\to 0$, see \cite[Section 4.2.2]{HMPS16}

\section{Steady states in 1D with $D=0$}
The system \eqref{eq:Poisson}--\eqref{eq:A} with $D=0$ in the spatially one-dimensional setting,
i.e., $\Omega:=(0,1)$, reads
\( \label{eq:1D1}
   -\partial_x ( (r+\A)\partial_x p ) &=& S, \\
   \label{eq:1D2}
   \partial_t \A - c^2 (\partial_x p)^2 + \frac{M'(|\A|)}{|\A|} \A &=& 0.
\)
Let us recall that the weak solution, constructed as the gradient flow in Proposition \eqref{prop:GF},
remains nonnegative for almost all $t>0$ and $x\in(0,1)$.
Denoting $B(x):=-\int_0^x S(y) \d y$ and integrating \eqref{eq:1D1} on the interval $(0,x)$,
taking into account the homogeneous Neumann boundary condition $\partial_x p(0)=0$, we obtain
\(  \label{eq:1D3}
   \partial_x p = \frac{B(x)}{r+\A}.
\)
Note that due to the assumption \eqref{ass:S}, we have $B(1) = - \int_0^1 S(y) \d y = 0$,
and, consequently, $\partial_x p(1) = 0$, so that the homogeneous Neumann boundary
condition for $p$ is satisfied at $x=1$ as well.

Inserting \eqref{eq:1D3} into \eqref{eq:1D2}, we obtain
\(  \label{eq:1D4}
      \partial_t \A = \frac{c^2 B(x)^2}{ (r+\A)^2 } - \frac{M'(|\A|)}{|\A|} \A.
\)
For further discussion, let us again adopt the generic choice $M(s)=s^\gamma/\gamma$
with the metabolic exponent $\gamma>0$.
We first study the case $\gamma\neq 1$.
Then, $\frac{M'(|\A|)}{|\A|} \A = |\A|^{\gamma-2}\A$, and recalling that $\A\geq 0$, we have $\frac{M'(|\A|)}{|\A|} \A = \A^{\gamma-1}$.
Therefore, steady states of \eqref{eq:1D4} are characterized as solutions of the algebraic equation
\(  \label{eq:1D5}
    (r+\A)^2 \A^{\gamma-1} = c^2 B(x)^2.
\)
A simple calculation reveals that if $\gamma>1$ the left-hand side of \eqref{eq:1D4}
is a strictly increasing function of $\A$, vanishing at $\A=0$. Therefore, for a given profile $B=B(x)$,
there exists a unique steady state $\A=\A(x)$ of \eqref{eq:1D4} on $\Omega=(0,1)$,
obtained as the unique solution of \eqref{eq:1D5} for each $x\in(0,1)$.
Another simple calculation reveals that this solution is asymptotically stable.
It is instructive to compare this result with \cite[Section 6.1]{HMP15} for the one-dimensional
version of the model \eqref{eq01}--\eqref{eq02},
where for $\gamma>1$ one has the unstable trivial steady state $m=0$ and a pair
of stable states $\pm m_s\neq 0$.
%We see that the model \eqref{eq:1D4} does not induce the (unstable) trivial steady state.

For $\gamma\in(0,1)$ the left-hand side of \eqref{eq:1D4} has the unique minimum
$r_\gamma := \left(\frac{2}{1+\gamma}\right)^2 \left( \frac{1-\gamma}{1+\gamma} \right)^{\gamma-1} r^{\gamma+1}$
on $(0,\infty)$. Consequently:
\begin{itemize}
\item
If $c^2 B(x)^2 > r_\gamma$, then \eqref{eq:1D4} has two positive steady states $0 < \A_1 < \A_2$,
where $\A_1$ is unstable and $\A_2$ is asymptotically stable.
\item
If $c^2 B(x)^2 = r_\gamma$, then \eqref{eq:1D4} has one steady state, asymptotically stable from the right-hand side
and unstable from the left.
\item
If $c^2 B(x)^2 < r_\gamma$, then \eqref{eq:1D4} has no steady state and its solution,
starting from a positive initial value, decays to zero in finite time.
\end{itemize}
Let us note that when $\gamma\in(0,1)$ and $\A$ reaches zero, the right-hand side of \eqref{eq:1D5} blows up and the equation loses validity.
From the modeling point of view, it makes sense to extend the solution $\A$ by zero past this point;
the transport of material in the (one-dimensional) structure is then due to the background permeability $r>0$ only.

For $\gamma=1$ equation \eqref{eq:1D5} is to be interpreted as the inclusion
\[
   \frac{c^2 B(x)^2}{ (r+\A)^2 } \in \sign(\A),
\]
where
\[
   \sign(\A) =  \left\{ \begin{array}{ll}
      %\displaystyle 
      \{1\} & \textrm{for } \A>0, \\[4mm]
       [0,1] & \textrm{for } \A=0.
  \end{array} \right.
\]
Obviously, $\A(x)>0$ is only possible if $c|B(x)| > r$. Consequently, the steady state is of the form
$\A(x) = \left(c|B(x)| - r\right)^+$, where $x^+ := \max\{0,x\}$. %, and is asymptotically stable.
Note that the total permeability is then $\A(x) + r = \max\{c|B(x)|, r\}$ so that
in the region where $c|B(x)| < r$ the transport of material is by the background permeability $r>0$
of the medium.
Again, this may be compared with \cite[Section 6.1]{HMP15} for the model \eqref{eq01}--\eqref{eq02},
where for $\gamma=1$ there is a pair of nonzero stable steady states $\pm m_s\neq 0$ if and only if
$c|B(x)|$ is large enough. On the other hand, for small $c|B(x)|$ one has only the trivial steady state $m=0$,
which is stable.

%%%%%%%%%%%%%%%%%%%%%%%%%%%%%%%%%
\section*{Acknowledgments}
G. P. acknowledges support from the Austrian Science Fund (FWF)
through the grants F 65 and W 1245.
J. H. acknowledges the fruitful discussions with Oliver Tse that have taken place
during the author's visit of TU Eindhoven, which helped to initiate some
ideas presented in this paper.

%%%%%%%%%%%%%%%%%%%%%%%%%%%%%%%%%%%%%%%%%%%%%%%%%%%%%%


\begin{thebibliography}{99}

\bibitem{AAFM} G. Albi, M. Artina, M. Fornasier, P. Markowich: \emph{Biological transportation networks: modeling and simulation.}
Analysis and Applications. Vol. 14, Issue 01 (2016).

\bibitem{bookchapter}
G. Albi, M. Burger, J. Haskovec, P. Markowich, M. Schlottbom: \emph{Continuum Modelling of Biological Network Formation.}
In: N. Bellomo, P. Degond, and E. Tamdor (Eds.), 
{\em Active Particles Vol.I - Theory, Models, Applications}, Series: Modelling and Simulation in Science and Technology, Birkh\"auser-Springer (Boston), 2017.
%{\tt doi:} 10.1007/978-3-319-49996-3.

\bibitem{AGS}
L. Ambrosio, N. Gigli, G. Savar\'{e}: \emph{Gradient Flows in Metric Spaces and in the Space of Probability Measures}.
Birkh\"auser Verlag Basel -- Boston -- Berlin, 2005.

\bibitem{BHMR}
M. Burger, J. Haskovec, P. Markowich, H. Ranetbauer: \emph{A mesoscopic model of biological transportation networks.}
Preprint, {\tt arXiv:1806.00120v2} (2018).

\bibitem{Brezis}
H. Brezis: \emph{Functional Analysis, Sobolev Spaces and Partial Differential Equations.}
Springer (2011).

\bibitem{Ciarlet}
P. G. Ciarlet: \emph{The finite element method for elliptic problems.}
North-Holland, Amsterdam (1978).

\bibitem{Evans}
L.C. Evans: \emph{Partial differential equations.}
Graduate Studies in Mathematics Vol. 19, American Mathematical Society (2010).

%\bibitem{Evans-Gariepy}
%L.C. Evans, R.F. Gariepy: \emph{Measure Theory and Fine Properties of Functions.}
%CRC press (2015).

\bibitem{HMP15} J. Haskovec, P. Markowich,  B. Perthame: \emph{Mathematical Analysis of a PDE System for Biological Network Formation.}
Comm. PDE 40:5, pp. 918-956 (2015).

\bibitem{HMPS16} J. Haskovec, P. Markowich, B. Perthame,  M. Schlottbom: \emph{Notes on a PDE system for biological network formation.}
Nonlinear Analysis 138 (2016), pp. 127--155.  %doi: 10.1016/j.na.2015.12.018

\bibitem{HMP19} J. Haskovec, P. Markowich, G. Pilli: \emph{Murray's law for discrete and continuum models of biological networks.}
Math. Mod. Meth. Appl. Sci. 29 (2019), pp. 2359-2376.
{\tt doi:} 10.1142/s0218202519500489.

\bibitem{Gilbarg-Trudinger}
D. Gilbarg and N.S. Trudinger: \emph{Elliptic Partial Differential Equations of Second Order.}
Grundlehren der mathematischen Wissenschaften Vol. 224, Springer (1998).

\bibitem{HKM}
J. Haskovec, L. M. Kreusser,  P. Markowich: \emph{ODE and PDE based modeling of biological transportation networks}.
arXiv:{1805.08526} (2018).

\bibitem{HKM2}
J. Haskovec, L. M. Kreusser,  P. Markowich: \emph{Rigorous continuum limit for the discrete network formation problem.}
arXiv: (2018).

\bibitem{Hu} D. Hu: \emph{Optimization, Adaptation, and Initialization of Biological Transport Networks.}
Notes from lecture (2013).

\bibitem{Hu-Cai} D. Hu, D. Cai: \emph{Adaptation and Optimization of Biological Transport Networks.}
Phys. Rev. Lett. 111 (2013), 138701.

\bibitem{Hu-Cai-19}
D. Hu and D. Cai: \emph{An optimization principle for initiation and adaptation of biological transport networks.}
Comm. Math. Sci. 17, 5, pp. 1427 (2019).

\bibitem{Li}
B. Li:  \emph{Uniqueness of classical solution to an elliptic-parabolic system in biological network formation.}
Journal of Physics: Conference Series.(2018).
 1053. 012024. 10.1088/1742-6596/1053/1/012024. 

\bibitem{RogerJohnson}
Horn, Roger A., and Johnson, Charles R: \emph{Matrix analysis}. Cambridge Univ. Press. (1990). 

\bibitem{Rourke}
J. O'Rourke, S.L. Devadoss: \emph{Discrete and Computational Geometry}.
First ed. (2011), Princeton University Press.

\bibitem{Santambrogio}
F. Santambrogio: \emph{\{Euclidean, metric, and Wasserstein\} gradient flows: an overview.}
Bulletin of Mathematical Sciences 7 (2017), pp. 87--154.

%\bibitem{Simon87} J. Simon: \emph{Compact sets in the space $L^p(0,T; B)$.}
%Ann. Mat. Pure Appl. IV (146), 1987, pp. 65-96.

\bibitem{Xu2} 
Xu, Xiangsheng. \emph{
	Regularity theorems for a biological network formulation model in two space dimensions}. Kinetic and Related Models,11, 2, pp. 397--408 (2018).

\bibitem{Xu}
Xu, Xiangsheng. \emph{Partial regularity of weak solutions and life-span of smooth solutions to a biological network formulation model}.
SN Partial Differential Equations and Applications 1 (2020).
%https://doi.org/10.1007/s42985-020-00021-3(0123456789().,-volV)(012345678


\end{thebibliography}
\end{document}